\documentclass{birkjour}

\usepackage{
amssymb,
amsfonts,
amsmath,
mathtools,
amscd,
amsthm,
commath,
esdiff
}

\usepackage{stmaryrd}
\usepackage{textcomp}

\usepackage[pagewise]{lineno}
\usepackage[utf8]{inputenc}	
\usepackage{hyperref, color}
\hypersetup{
	colorlinks = true,
	linkcolor = red,
	filecolor = blue,
	urlcolor = blue,
	citecolor = blue,
    pdftitle = {Domingos-Onnis-Piu}
}

\usepackage{url}
\usepackage{epsfig}

\usepackage{pstricks}
\usepackage{pst-plot}
\usepackage{pst-node}

\def \r{\mathbb R}

\def \r{\mbox{${\mathbb R}$}}

\def \N{\mbox{$\mathcal{N}$}}

\usepackage{marginnote}

\newtheorem{theorem}{Theorem}[section]
\newtheorem*{theorem*}{Theorem}

\newtheorem*{lemma*}{Lemma}

\newtheorem{taggedtheoremx}{Theorem}
\newenvironment{taggedtheorem}[1]
 {\renewcommand\thetaggedtheoremx{#1}\taggedtheoremx}
 {\endtaggedtheoremx}

\theoremstyle{definition}

\newtheorem{example}[theorem]{Example}
\newtheorem{remark}[theorem]{Remark}

\numberwithin{equation}{section}
\numberwithin{figure}{section}

\begin{document}

\title[The Bour's theorem for invariant surfaces]{The Bour's theorem for invariant surfaces\\ in three-manifolds}

\author[I. Domingos]{Iury Domingos}

\address{KU Leuven\\
	Department of Mathematics\\
	Celestijnenlaan 200B -- Box 2400, 3001  Leuven, Belgium}

\email{iury.domingos@im.ufal.br}

\author[I. I. Onnis]{Irene I. Onnis}

\address{Universit\`a degli Studi di Cagliari\\
Dipartimento di Matematica e Informatica\\
Via Ospedale 72, 09124 Cagliari, Italy}

\email{irene.onnis@unica.it}

\author[P. Piu]{Paola Piu}
\address{Universit\`a degli Studi di Cagliari\\
Dipartimento di Matematica e Informatica\\
Via Ospedale 72, 09124 Cagliari, Italy}

\email{piu@unica.it}

\thanks{I.I. Onnis and P. Piu were supported by Fondazione di Sardegna under Project GoAct CUP F75F21001210007 and partially funded by PNRR e.INS Ecosystem of Innovation for Next Generation Sardinia (CUP F53C22000430001, codice MUR ECS00000038). I. Domingos was supported by the Research Foundation-Flanders (FWO) and the Fonds de la Recherche Scientifique (FNRS) under EOS Project G0H4518N}



\begin{abstract}
 In this paper, we apply techniques from equivariant geometry to prove that a generalized Bour's theorem holds for surfaces that are invariant under the action of a one-parameter group of isometries of a three-dimensional Riemannian manifold.
\end{abstract}

\maketitle

\section{Introduction}

The starting point of this work is a classical result by Edmond Bour concerning helicoidal surfaces in Euclidean space $\r^3$ (see \cite[Theorem~II, p. 82]{Bour-1863}).  In 1862, Bour demonstrated the existence of a two-parameter family of helicoidal surfaces that are isometric to a given helicoidal surface in $\mathbb{R}^3$. For this, firstly he obtained orthogonal parameters $(s,t)$ on a helicoidal surface $M$ for which the families of $s$-coordinate curves are geodesics on $M$, parametrized by arc length, and the $t$-coordinate curves are the trajectories of the helicoidal motion. Such parameters are called {\it natural parameters} and the first fundamental form with respect to them can be written as $\dif\sigma^2=\dif s^2+U(s)^2\dif t^2$. On the contrary, given the natural parameters $(s,t)$ on $M$  and the function $U(s)$, Bour determined a two-parameter family of isometric helicoidal surfaces that have induced metric given by  $\dif\sigma^2=\dif s^2+U(s)^2\dif t^2$.

By using the Bour's theorem, do Carmo and Dajczer established in \cite{doCarmo-Dajczer-82} a condition for a helicoidal surface that belongs to the Bour's family to have constant mean curvature. Also they obtained an integral representation (depending on three parameters) of helicoidal surfaces with nonzero constant mean curvature (CMC), that is a natural generalization of the representation for Delaunay surfaces, i.e. CMC rotation surfaces, given by Kenmotsu in \cite{Kenmotsu-1980}.

In other three-dimensional Riemannian manifolds, some versions of the Bour's theorem were established in the last decades. In space forms, this result is due to Ordóñes \cite{Ordones-1995} and, as a consequence,  all CMC helicoidal surfaces in these spaces were classified, via its profile curves. In \cite{Ikawa-01v2}, Ikawa proved this result for helicoidal surfaces in Minkowski space $\mathbb{R}^3_1$ with rotation axis being spacelike, timelike or null; Ji and Kim proved that it holds for cubic screw motion surfaces in $\mathbb{R}^3_1$, i.e., for surfaces invariant by a non-trivial one-parameter families of translations together with a Lorentzian rotation around a null axis (cf. \cite{Ji-Kim-2010}, see also \cite{Dillen-Kuhnel-99}).

In \cite{SaEarp-08,saearp2005}, the authors generalized the Bour's theorem for $\mathbb{S}^2\times\mathbb{R}$, $\mathbb{H}^2\times\mathbb{R}$ and the Heisenberg space. In the case of $\mathbb{H}^2\times\mathbb{R}$, it has been shown that it holds in three different scenarios, wherein the surface is invariant by a one-parameter group of isometries, referred as either standard, or parabolic, or hyperbolic, screw motion surfaces, respectively (indeed, the study of the invariant surfaces in $\mathbb{H}^2\times\mathbb{R}$ can be reduced to these three types, up to congruences, as has been proved by the second author in \cite[Proposition~2]{Onnis-2008}). As an application of these results, a two-parameter family of complete embedded, simply connected, minimal and CMC surfaces in $\mathbb{H}^2\times\mathbb{R}$ was found.

In \cite{Caddeo-Onnis-Piu-2022}, the second and the third authors generalize the Bour's theorem for helicoidal Bianchi-Cartan-Vranceanu (BCV) spaces, i.e. in the Riemannian $3$-manifolds whose metrics have groups of isometries of dimension $4$ or $6$, except the hyperbolic one. Using a common setting for the BCV spaces, they proved that there exists a two-parameter family of helicoidal surfaces isometric to a given helicoidal surface; moreover, by making use of this two-parameter representation, they characterized the helicoidal surfaces which have constant mean curvature, including the minimal ones. These results generalize the ones proved by do Carmo and Dajczer in \cite{doCarmo-Dajczer-82}.

In a broader sense, the surfaces which are invariant under the action of a one-parameter subgroup of isometries of the ambient space constitute an important geometric class of surfaces. The foundations of equivariant geometry, which refers to the study of interactions between transformation groups and Riemannian geometry, has been initiated by Back, do Carmo, and Hsiang in a work from the 1980s. However, the publication of their work was delayed until 2009 (cf. \cite{Back-doCarmo-Hsiang-2009}).

The invariant surfaces have also been classified by Gaussian or mean curvature in many remarkable three-dimensional spaces. By using techniques from equivariant geometry,
Gomes studied the CMC spherical surfaces in the hyperbolic 3-space in \cite{Gomes-1987}. In the BCV spaces, Caddeo, Piu and Ratto in \cite{Caddeo-Piu-Ratto-95} characterized the $\mathrm{SO}(2)$-invariant CMC surfaces, and later in \cite{Caddeo-Piu-Ratto-1996}, they classified rotational surfaces in the Heisenberg space with constant Gaussian curvature. A complete description of CMC invariant surfaces in the Heisenberg space and in $\mathbb{H}^2\times\mathbb{R}$ were established by Figueroa, Mercuri and Pedrosa in \cite{Figueroa-Mercuri-Pedrosa-99}, and by Montaldo the second author in \cite{Montaldo-Onnis-2004,Onnis-2008}, respectively. Furthermore, in an arbitrary three-dimensional manifold, invariant surfaces were also studied in \cite{Mercuri-Montaldo-Onnis-2011,Montaldo-Onnis-2005}; and in \cite{Montaldo-Onnis-2009,Montaldo-Onnis-2011,Piu-Profir-2011} have been studied two well-known types of curves on invariant surfaces: the geodesics and the proper biharmonic curves.

The aim of this work is to prove that the Bour's theorem holds for all surfaces which are invariant under the action of a one-parameter group of isometries of the ambient space, namely:

\begin{theorem*}
 [Generalized Bour's theorem] Let $(N,g)$ be a three-dimensional connected manifold Riemannian, which admits a Killing vector field $X$, and let $M$ be a $G_X$-invariant surface of $N$. Then there are local natural parameters $(s,t)$ on $M$ such that the induced metric is given by
\[\dif\sigma^2=\dif s^2+U(s)^2\dif t^2,\]
where $U(s)$ is a positive smooth function. Furthermore, there exists a one-parameter family of $G_X$-invariant surfaces of $N$ isometric to $M$, such that each surface in this family can be explicitly determined in terms of a parameter $m$ and the metric $\dif \sigma^2$, up to an integration procedure.
\end{theorem*}

This work is organized as follows: In Section~\ref{sec:preliminaries}, we fix some notations and we recall previous results about the theory of Riemannian actions, in particular the Principal orbit theorem. We introduce the equivariant immersions, by using standard techniques of equivariant geometry developed by Back, do Carmo and Hsiang in \cite{Back-doCarmo-Hsiang-2009}. Also, in the case of surfaces invariant by a one-parameter subgroup of isometries generated by a Killing vector field, we give a local description of these immersions in terms of the flow of the Killing vector field.

In Section~\ref{sec:natural-parameters}, we construct the natural parameters on an invariant surface in a general Riemannian three-manifold, generalizing the classical ones constructed by Bour for the helicoidal surfaces in $\mathbb{R}^3$. By using these parameters, we are able to show the existence of the Bour's family in the Theorem~\ref{Generalized-Bour}.

In Section~\ref{sec:integration-procedure}, we use the Frobenius theorem to construct orthogonal coordinates on the orbit space, associated to the action of Killing vector field. We develop an integration procedure that allows to find explicitly the surfaces of the Bour's family in Theorem~\ref{Generalized-Bour-2}. Finally, by using this method, we obtain the classical Bour's theorem for helicoidal surfaces in $\mathbb{R}^3$ and in the BCV spaces, as presented in \cite{Bour-1863,doCarmo-Dajczer-82} and \cite{Caddeo-Onnis-Piu-2022}, respectively.

\subsection*{Acknowledgements} I. Domingos would like to thank the hospitality of the Mathematics Department of the Universit\`a degli Studi di Cagliari, where part of this work was carried out.

\section{Preliminaries}\label{sec:preliminaries}

Let $(N,g)$ be a connected Riemannian manifold and let $X$
be a Killing vector field on ${N}$. Then $X$ generates a
one-parameter subgroup $G_X$ of the group of isometries of
$(N,g)$. From the theory of Riemannian actions (we may refer to \cite{Bredon-72} for more details) we know that:
\begin{itemize}
  \item The isotropy subgroup $G_p$, $p\in N$, is compact and the orbit $G(p)$ is diffeomorphic to the quotient space $G_X/G_p$. We say that $G(q)$ is of {\it the same type} as $G(p)$ if the isotropy subgroups $G_q$ and $G_p$ are conjugated.
  \item  An orbit $G(p)$ is called {\it principal} if there exists an open neighborhood  $U\subset N$ of $p$
such that all orbits $G(q)$, $q\in U$, are of the same type as $G(p)$. This implies that $G(q)$  is
diffeomorphic to $G(p)$. We denote with $N_r$ the {\it regular part} of $N$, that is, the subset consisting of points belonging to
principal orbits.
\item An orbit $G(p)$ is said to be {\it exceptional} if the dimension of $G(p)$ coincides with the dimension of principal orbits but $G(p)$ is not a principal orbit.
  \item An orbit $G(p)$ is said to be {\it singular} if  has dimension less than the dimension of the principal orbits.
\end{itemize}

The following theorem (cf. \cite[Theorem~3.1, p.179]{Bredon-72}) is a very important result for the study of the orbit space of a group action:
 \begin{theorem*}
 [Principal orbit theorem]  Let G be a compact Lie group acting isometrically on a Riemannian manifold $N$. The following hold:
 \begin{enumerate}
\item All principal orbits are diffeomorphic,
\item $N_r$ is open and dense in $N$,
\item the regular part of the orbit space $N_r/G_X$ is a connected differentiable manifold and the quotient map
$\pi:N_r\to N_r/G_X$ is a submersion.
\end{enumerate}
\end{theorem*}

From now on, we suppose that  $(N,g)$ a three-dimensional connected Riemannian manifold. Let $f : M \to (N,g)$ be an immersion from a
surface $M$ into $N$ and assume that $f({M})\subset {N}_r$
(the regular part of $N$,  that is, the subset consisting of
points belonging to principal orbits). We say that $f$ is a $G_X$-{\it
equivariant immersion}, and $f({M})$ a $G_X$-{\it invariant surface}
of ${N}$, if there exists an action of $G_X$ on ${M}$ such that
for any $x\in {M}$ and $a\in G_X$ we have $f(a\,x)=a f(x)$.
A $G_X$-equivariant immersion $f:M\to (N,g)$ induces on ${M}$ a Rieman\-nian metric,
the pull-back metric, denoted by $g_f$ and called the $G_X$-{\it invariant induced metric}.

Consider $f:M\to (N,g)$ a $G_X$-equivariant immersion and let $g_f$ be the
 $G_X$-invariant induced metric on $M$.
Assume
that ${N}/G_X$ is connected. Then, $f$ induces an
immersion $\tilde{f}:{M}/G_X\rightarrow {N}_r/G_X$ between the orbit
spaces and, moreover, ${N}_r/G_X$ can be equipped with a
Riemannian metric, the {\it quotient metric}, so that the quotient
map $\pi:{N}_r\to {N}_r/G_X$ becomes  a Riemannian submersion.
Thus we have the following diagram
\[\begin{CD}
(M,g_f) @>f>> ({N}_{r},g)\\
@V VV @V\pi VV\\
M/G_X @>\tilde{f}>> ({N}_{r}/G_X,\tilde{g}).
\end{CD}\]

For later use, we describe the quotient metric of the regular part of the orbit space
$N_r/G_X$. It is well-known that
$N_r/G_X$ can be locally parametrized by invariant functions of
the Killing vector field $X$. In addition, if $\{\xi_1,\xi_2\}$ is a complete set of
invariant functions on a $G_X$-invariant subset  of $N_r$, then the
quotient metric is given by  $\tilde{g}=\sum_{i,j=1}^{2} h^{ij}
\dif\xi_i\otimes \dif\xi_j$, where $(h^{ij})$ is the inverse of the matrix
$(h_{ij})$ with entries $h_{ij}=g(\nabla \xi_i,\nabla \xi_j)$ and $\nabla \xi_i$ is the gradient vector field of $\xi_i$ on $N$ (we may refer to \cite{Hsiang-85,Olver-93} for more details).

Let us now provide a local description of the $G_X$-invariant surfaces of
$N$. Let $\tilde{\gamma}:(a,b)\subset\r\to(N/G_X,\tilde{g})$
be a curve and
$\gamma:(a,b)\subset\r\to N$ be a lift of $\tilde{\gamma}$, such
that $\dif\pi(\gamma')=\tilde{\gamma}'$. If we denote by
$\phi_v,\;v\in(-\epsilon,\epsilon)$, the local flow of the Killing
vector field $X$, then the map
\begin{equation}\label{eq-psi}
\psi:(a,b)\times(-\epsilon,\epsilon)\to N\,,\quad \psi(u,v)=\phi_v(\gamma(u)),
\end{equation}
defines a parametrization of a
$G_X$-invariant surface. Conversely, if $f({M})$ is a $G_X$-invariant immersed surface in $N$,
then $\tilde{f}$ defines a curve $\tilde{\gamma}$ in $(N/G_X,\tilde{g})$, generally called the {\it profile
curve} of the invariant surface.

As the $v$-coordinate curves are the orbits of the
action of the one-parameter group of isometries $G_X$, it results that the
coefficients of the pull-back metric $g_\psi=E\dif u^2 +2 F\dif u \dif v + G\dif v^2$ are functions only of $u$ and are given by:
$$
\left\{\begin{aligned}
E(u)&=g(\psi_u,\psi_u)=g(\dif\phi_v(\gamma'(u)),\dif\phi_v(\gamma'(u)))\\&=g(\gamma'(u),\gamma'(u)),\\
F(u)&=g(\psi_u,\psi_v)=g(\dif\phi_v(\gamma'(u)),X(\gamma(u)))\\&=g(\gamma'(u),X(\gamma(u))),\\
G(u)&=g(\psi_v,\psi_v)\\&=g(X(\gamma(u)),X(\gamma(u))).
\end{aligned}
\right.
$$
We observe also that the volume function  $\omega:=\|X(\gamma)\|_g$ (of the principal orbit) is a $X$-invariant function and, therefore, $\omega=\omega(\xi_1,\xi_2)$.

\section{The natural parameters of an invariant surface}\label{sec:natural-parameters}

In order to prove the generalized Bour's theorem, we begin by constructing orthogonal parameters $(s,t)$ on an invariant surface for which the families of $s$-coordinate curves are geodesics on $M$ parametrized by arc length, and the $t$-coordinate curves are the orbits of the Killing vector field $X$.

\begin{lemma*}
Let $(N,g)$ be a three-dimensional connected manifold Riemannian which admits a Killing vector field $X$. Then any $G_X$-invariant surface $M$ of $N$ with profile curve $\tilde{\gamma}$ can be locally parametrized by natural parameters $(s,t)$, so that its first fundamental form is given by
$$g_\psi=\dif s^2+ \omega(s)^2\dif t^2,$$
where $\omega=\|X(\gamma)\|_g$ is the volume function (of the principal orbit) restricted to a lift $\gamma$ of $\tilde{\gamma}$ with respect to $\pi$.
\end{lemma*}
\begin{proof}
We suppose that $M$ is locally parametrized by \eqref{eq-psi}.
Then,  the induced metric is given by
\begin{equation}\label{metric}
\begin{aligned} g_\psi&=E(u)\dif u^2+2F(u)\dif u\dif v+G(u)\dif v^2\\
&=\bigg(E(u)-\frac{F(u)^2}{G(u)}\bigg)\dif u^2+ G(u)\bigg(\dif v+\frac{F(u)}{G(u)}\dif u\bigg)^2,
\end{aligned}
\end{equation}
where $G(u)=\omega(u)^2$ and $\omega(u)=\|X(\gamma(u))\|_g$ is the volume function.

As $\gamma$ is a lift of $\tilde{\gamma}$ with respect to the Riemannian submersion $\pi$, we have that $\dif\pi(\psi_u)=\tilde{\gamma}'$ and
$\dif\pi(\psi_v)=0$. Let  $e$ be a local unit vector field tangent to the surface and horizontal with respect to $\pi$, then $\psi_u$ can be decomposed as
\[
\psi_u= g(\psi_u,X)\frac{X}{g(X,X)}+g(\psi_u,e)\,e = \frac{F(u)}{G(u)}\,X+g(\psi_u,e)\,e.
\]
Therefore,  \[E(u)=g(\psi_u,\psi_u)=g(\psi_u,e)^2+\frac{F(u)^2}{G(u)}.\]
Also, since $\dif\pi(\psi_u)=\tilde{\gamma}'$ and $\pi$ is a Riemannian submersion,  we have that
\begin{equation}\label{pc}
\begin{aligned}
\tilde{g}(\tilde{\gamma}',\tilde{\gamma}')&=g(\psi_u,e)^2\,\tilde{g}(d\pi(e),d\pi(e))=g(\psi_u,e)^2\,g(e,e)\\&=g(\psi_u,e)^2=E(u)-\frac{F(u)^2}{G(u)}.
\end{aligned}
\end{equation}
We introduce new parameters $(s, t)$ as follows:
\begin{equation*}\label{parameters}
\left\{
\begin{aligned}
s&=\int \sqrt{E(u)-\frac{F(u)^2}{G(u)}}\dif u,\\
t&=v+\int \frac{F(u)}{G(u)}\dif u.
\end{aligned}
\right.
\end{equation*}
Since the Jacobian $|\partial (s,t)/\partial(u,v)|$ is non-zero, then $(s,t)$ are local coordinates on $M$ and also we can write \eqref{metric} as
\begin{equation*}\label{metric1}
g_\psi=\dif s^2+\omega(s)^2\dif t^2,
\end{equation*}
where $\omega(s):=\omega(u(s))$ and $u(s)$ in the inverse of the arc length function of $\tilde{\gamma}(u)$.
We now observe that the $s$-coordinate curves,
given by
\[\psi(u(s),v(s,t_0))=\psi\Big(u(s),t_0-\int  \frac{F(u(s))\,u'(s)}{G(u(s))}\dif s\Big) \ \ \text{for} \ \ t_0\in\r,\]
are parametrized by arc length and also are orthogonal to the $t$-coordinate curves, i.e.  the orbits of the Killing vector field $X$.
Therefore these curves
are geodesics of $M$ (cf. \cite[Proposition~2.3]{Montaldo-Onnis-2011}) and consequently the local parametrization $\psi(u(s),v(s,t))$ is a {\it natural} parametrization of the invariant surface $M$.
\end{proof}

By using the above lemma, we can to state the first part of our main result:

\begin{taggedtheorem}{A}\label{Generalized-Bour}
Let $(N,g)$ be a three-dimensional connected manifold Riemannian, which admits a Killing vector field $X$, and let $M$ be a $G_X$-invariant surface of $N$ locally parametrized by natural parameters $(s,t)$ such that the metric is given by
\[\dif\sigma^2=\dif s^2+U(s)^2\dif t^2,\]
where $U(s)$ is a positive smooth function.
Then there exists a one-parameter family of $G_X$-invariant surfaces of $N$ isometric to $M$.
\end{taggedtheorem}

\begin{proof}
From \eqref{metric} and \eqref{pc},  the induced metric of a $G_X$-invariant surface  $\psi(u,v)=\phi_v(\gamma(u))$ with profile curve $\tilde{\gamma}(u)=\pi(\gamma(u))$ can be written as
\begin{equation*}\label{metric2}
g_\psi=||\tilde{\gamma}'(u)||_{\tilde{g}}^2\,\dif u^2+\omega(u)^2\,\Bigg[\dif v+\frac{g(X(\gamma(u)),\gamma'(u))}{\omega(u)^2}\, \dif u\Bigg]^2.
\end{equation*}
Therefore, we want to determine functions $u, v$ of $(s,t)$ such that
\begin{equation}\label{viceversa}
\left\{\begin{aligned}
\dif s&=||\tilde{\gamma}'(u)||_{\tilde{g}}\dif u,\\
\pm U(s)\dif t&=\omega(u)\,\Bigg[\dif v+\frac{g(X(\gamma(u)),\gamma'(u))}{\omega(u)^2}\dif u\Bigg].
\end{aligned}\right.
\end{equation}
From the first equation of \eqref{viceversa} it results that $u=u(s)$. Then, from the second one, we obtain
\[\left\{\begin{aligned}
\frac{\partial v}{\partial s}&=-u'(s)\,\frac{g(X(\gamma(u(s))),\gamma'(u(s)))}{\omega(u(s))^2}\\
&=-\frac{g(X(\gamma(s)),\gamma'(s))}{\omega(s)^2},\\
\frac{\partial v}{\partial t}&=\pm \frac{U(s)}{\omega(u(s))}=\pm \frac{U(s)}{\omega(s)},
\end{aligned}
\right.\]
where $\omega(s):=\omega(u(s))$ and $\gamma(s):=\gamma(u(s))$.
Therefore,
\[\frac{\partial^2 v}{\partial t\partial s}=0\]
and, hence, there exists a constant $m\neq 0$
such that
\begin{equation*}
\pm \frac{U(s)}{\omega(s)}=\frac{1}{m}.
\end{equation*}
Thus the second equation of system~\eqref{viceversa} becomes
\begin{equation*}\label{d}
\dif v=\frac{\dif t}{m}- \frac{g(X(\gamma(s)),\gamma'(s))}{m^2\,U(s)^2}\dif s.
\end{equation*}
So, we have the following system
$$\left\{\begin{aligned}
&\omega(s)=\pm m\, U(s),\\
&||\tilde{\gamma}'(s)||_{\tilde{g}}=1,
\end{aligned}
\right.$$
or equivalently
\begin{equation}\label{system}
\left\{\begin{aligned}
&\omega(\xi_1(s),\xi_2(s))=\pm m\, U(s),\\
&h^{11} (s)\,\xi_1'(s)^2+2h^{12} (s)\,\xi_1'(s) \xi_2'(s)+h^{22}(s) \,\xi_2'(s)^2=1,
\end{aligned}\right.
\end{equation}
where $h^{ij}(s)=h^{ij}(\xi_1(s), \xi_2(s))$ and $\tilde{\gamma}(s)=(\xi_1(s), \xi_2(s))$. Therefore, given a positive function $U(s)$ and an arbitrary $m\neq 0$, the $G_X$-invariant surface is locally parametrized by
\[
\psi(s,t)=\phi_{v(s,t)}(\gamma(s)),
\]
where
\[
v(s,t)=\frac{1}{m}\int \dif t-\int \frac{g(X(\gamma(s)),\gamma'(s))}{m^2\,U(s)^2}\dif s
\]
and $\gamma$ is a lift of the profile curve $\tilde{\gamma}$.  By construction, the induced metric is given by $g_\psi=\dif s^2+U(s)^2\dif t^2$, i.e., the $G_X$-invariant surface is isometric to $M$.
\end{proof}

\begin{remark}
The Bour's family constructed in the above theorem contains the surface we started with for $m=1$.
\end{remark}

\section{An integration procedure}\label{sec:integration-procedure}

In this section, we explain how to integrate the system \eqref{system} to obtain the one-parameter family of invariant surfaces of Theorem~\ref{Generalized-Bour}. As a consequence, we show that these surfaces do not depend of the lift of $\tilde{\gamma}$, and thus, we are able to describe entirely any invariant surface of this family in terms of the parameter $m$ and the metric $\dif \sigma^2$.

We begin by introducing local coordinates that are adapted to the action of the Killing vector $X$. For this, let $p\in N_r$ such that $X(p)\neq 0$. Then $X\neq 0$ in a neighborhood of $p$ and, by Frobenius theorem, we may choose a smooth local coordinates $(x_1,x_2,x_3)$ in a neighborhood $V\subset N_r$ of $p$, such that $X=\frac{\partial}{\partial x_3}$. Let $\{\dif x_1,\dif x_2,\dif x_3\}$ denote the dual basis of the coordinate fields $\{\frac{\partial}{\partial x_1},\frac{\partial}{\partial x_2},\frac{\partial}{\partial x_3}\}$ and $g_{ij} = g(\frac{\partial}{\partial x_i},\frac{\partial}{\partial x_j})$ be the coefficients of the Riemannian metric $g$ in these coordinates, then on $V$ we have
\[
g = \sum_{i,j=1}^{3}g_{ij} \dif x_i\otimes \dif x_j \ \ \text{and} \ \ X(g_{ij})=0,
\]
since $X$ is a Killing vector field; therefore $g_{ij}$ are independent of $x_3$, for all $i,j$. We empathize that a $X$-invariant function $\theta$ on $V$ is characterized by $\frac{\partial \theta}{\partial x_3}=0$, thus for simplicity we write $\theta = \theta(x_1,x_2)$. Moreover, denoting by $\phi_v,\;v\in(-\epsilon,\epsilon)$, the local flow of the Killing vector field $X$ on $V$, the orbits of $X$ have the form
\[
v\mapsto (x_1,x_2,x_3+v), \ \text{for all} \ (x_1,x_2,x_3)\in V,
\]
and therefore the local flow of $X$ is just a translation in the $x_3$-direction. Henceforth, we refer to $(x_1,x_2,x_3)$ as \emph{$X$-adapted local coordinates} on $N$.

By \cite[Theorem~2.17]{Olver-93}, the coordinates on $V/G_X$ are provided by the complete set of $X$-invariant functions $\{\xi_1(x_1,x_2) = x_1,\xi_2(x_1,x_2) = x_2\}$, and in this sense, the Riemannian submersion $\pi:V\to V/G_X$ is nothing more than $\pi(x_1,x_2,x_3) = (x_1,x_2)$, if we identify $\xi_i$ with the coordinate function $x_i$ on $V$. In addition, we have that $g(\nabla \xi_i,\nabla \xi_j)=g^{ij}$ and then, by \cite[Proposition~1]{Hsiang-85}, the quotient metric is given by
\[
\widetilde{g} = \frac{1}{g^{11}g^{22}-(g^{12})^2}
\left(g^{22}\dif x_1^2-2g^{12}\dif x_1\otimes\dif x_2+g^{11}\dif x_2^2\right),\]
up to identification of $\xi_i$ with $x_i$.

Firstly, suppose that the Killing vector field $X$ has constant length, so up to scaling we can suppose that $\omega=1$. In this case the orbits of $X$ are geodesics of $(N,g)$ and the $G_X$-invariant surfaces are foliated by geodesics of the ambient space, i.e. they are ruled surfaces.  From the above, given a flat $G_X$-invariant surface $M$, it results that $U(s) = 1$ and $m=1$, since the Bour's family must contain this surface. Therefore, given a profile curve $\tilde{\gamma}(s) = (x_1(s),x_2(s))$, parametrized by arc length in the orbit space, and a lift $\gamma(s)=(x_1(s),x_2(s),x_3(s))$, we have that
\[
\psi(s,t)=\left(x_1(s),x_2(s),x_3 (s)+v(s,t)\right),
\]
with
\[
x_3(s)+v(s,t)=\int \dif t- \sum_{i=1}^{2}\int x_i'(s) \,g_{i3}(x_1(s),x_2(s))\dif s,\]
whose the induced metric is given by $\dif s^2+\dif t^2$. Moreover, we point out that $x_3(s)+v(s,t)$ does not depend of the lift of $\tilde{\gamma}$. In this case, the Bour's family consists of only one surface.

In the case that the Killing vector field $X$ does not have constant length, i.e., when the volume function $\omega$ is not a constant function, we argue in local coordinates to construct orthogonal coordinates on $N/G_X$, in which the volume function $\omega$ is one of them.

\subsection{Local orthogonal coordinates on the orbit space}

Suppose that $\omega$ is not a constant function and let $p\in N_r$ such that $X(p)\neq 0$. We may choose $X$-adapted local coordinates $(x_1,x_2,x_3)$ in a neighborhood $V\subset N_r$ of $p$, such that $\|\nabla\omega\|_g \neq 0$. The volume function, that can be written as $\omega = \sqrt{g_{33}}$, is a $X$-invariant function and thus we compute
\[
\nabla \omega = \sum_{i,j=1}^{2}g^{ij}\diffp{\omega}{{x_i}}\diffp{}{{x_j}}-\frac{1}{\omega}\Gamma_{33}^3\diffp{}{{x_3}},
\]
since $X$ is a Killing vector field. By shrinking $V$ if necessary, we can find a non-constant smooth function $\theta$ on $V$ such that $\|\nabla \theta\|_g\neq 0$ and
\begin{equation}\label{system-f}
\left\{\begin{aligned}
&\left(g^{11}\diffp{\omega}{{x_1}}+g^{12}\diffp{\omega}{{x_2}}\right)\diffp{\theta}{{x_1}}
+\left(g^{12}\diffp{\omega}{{x_1}}+g^{22}\diffp{\omega}{{x_2}}\right)\diffp{\theta}{{x_2}}=0, \\
& \diffp{\theta}{{x_3}}=0,
\end{aligned}\right.
\end{equation}
holds, that is, the solution $\theta$ of the first order linear partial differential equations system above is a $X$-invariant function on $V$ satisfying $g(\nabla \omega,\nabla \theta)=0$. We point out that $\theta$ is not unique, however it can be uniquely determined by $g$ and $X$ in the presence of a Cauchy data (see, for example, \cite{Bernstein-50}).

We claim that $(\omega,\theta)$ are local orthogonal coordinates on $V/G_X$. Indeed, let $\Psi:V\to\r^2$ given by $\Psi(x_1,x_2,x_3)=(\omega(x_1,x_2),\theta(x_1,x_2))$. We observe that $\textrm{rank}(\dif\Psi)=2$ at each point of $V$, since $g(\nabla \omega,\nabla \theta)=0$ and
\[
\left(g^{11}g^{22}-(g^{12})^2\right)\left(\diffp{\omega}{{x_1}}\diffp{\theta}{{x_2}}-
\diffp{\omega}{{x_2}}\diffp{\theta}{{x_1}}\right)^2 = \|\nabla\omega\|_g^2\|\nabla\theta\|_g^2\neq 0.
\]
Hence, by \cite[Theorem~2.16]{Olver-93}, $\{\omega,\theta\}$ is a set of functionally independent invariants on a $G_X$-invariant subset of $N_r$. As $\Psi$ is constant on the fibers of $\pi$, there exists a unique smooth map $\tilde{\Psi}:V/G_X\to\mathbb{R}^2$ such that $\Psi = \tilde{\Psi}\circ\pi$ (cf. \cite[Theorem~4.30]{Lee-13}), and thus $\textrm{rank}(\dif \widetilde{\Psi})=2$. Therefore, by the Implicit function theorem, we have that $x_1=x_1(\omega,\theta)$ and $x_2=x_2(\omega,\theta)$, i.e., $(\omega,\theta)$ are local coordinates on $V/G_X$, reducing the neighborhood if necessary. Moreover, since $\|\nabla \omega\|_g^2 = \|\nabla \omega\|_g^2(\omega,\theta)$,  $\|\nabla \theta\|_g^2 = \|\nabla \theta\|_g^2(\omega,\theta)$ and $g(\nabla \omega,\nabla \theta)=0$, again by \cite[Proposition~1]{Hsiang-85}, we get that the quotient metric $\tilde{g}$ is given by
\[
\tilde{g}=\frac{1}{\|\nabla \omega\|_g^2}
\dif \omega^2+\frac{1}{\|\nabla \theta\|_g^2}\dif \theta^2.
\]
In the local coordinates $(\omega,\theta)$, the system \eqref{system} takes the form
\begin{equation*}
\left\{\begin{aligned}
&\omega(s)=\pm m\, U(s),\\
&\frac{1}{\|\nabla \omega\|_g^2(s)}\omega'(s)^2+\frac{1}{\|\nabla \theta\|_g^2(s)}\theta'(s)^2=1,
\end{aligned}\right.
\end{equation*}
where $\|\nabla \omega\|_g^2(s) = \|\nabla \omega\|_g^2(\omega(s),\theta(s))$ and $\|\nabla \theta\|_g^2(s) = \|\nabla \theta\|_g^2(\omega(s),\theta(s))$. Since $U(s)$ is known, by the first equation, $\|\nabla \omega\|_g^2$ and $\|\nabla \theta\|_g^2$ can be seen as functions on $(s,\theta(s))$, and thus we can write the system as
\begin{equation}\label{omega-theta}
\left\{\begin{aligned}
&\omega(s)=\pm m\, U(s),\\
&M(s,\theta(s))+N(s,\theta(s))\theta'(s)=0,
\end{aligned}\right.
\end{equation}
where
\begin{equation*}
\left\{\begin{aligned}
 M(s,\theta(s)) &=  \|\nabla \theta\|_g(\pm m U(s),\theta(s))\sqrt{\|\nabla \omega\|_g^2(\pm m U(s),\theta(s))-m^2 U'(s)^2},\\
 N(s,\theta(s)) &=  \pm\|\nabla \omega\|_g(\pm m U(s),\theta(s)).
\end{aligned}\right.
\end{equation*}
Therefore $\theta(s)$ is a solution of an ordinary differential equation that can be solved by the Euler method. Since the profile curve of the $G_X$-invariant surface is $\tilde{\gamma}(s)=(\omega(s), \theta(s))$ on the coordinates $(\omega,\theta)$, then
\[
\tilde{\gamma}(s)=\widetilde{\Psi}^{-1}\left(\omega(s), \theta(s)\right)=\left(x_1(s),x_2(s)\right),
\]
where $x_1(s) = x_1(\omega(s), \theta(s))$, $x_2(s) = x_2(\omega(s), \theta(s))$ and $(\omega(s),\theta(s))$ satisfies system \eqref{omega-theta}. By taking a lift $\gamma(s)=(x_1(s),x_2(s),x_3(s))$ of $\tilde{\gamma}$, we compute
\[
g(X(\gamma(s)),\gamma'(s)) = \sum_{i=1}^{2}x_i'(s)\, g_{i3}(x_1(s),x_2(s))+x_3'(s)\,\omega(s)^2.
\]
Therefore, on the neighborhood $V$, the one-parameter family of $G_X$-invariant surfaces can be parametrized, in natural parameters, by
\begin{equation}\label{psi(s,t)}
\psi_m(s,t)=\left(x_1(s),x_2(s),x_3 (s)+v(s,t)\right),
\end{equation}
where
\[
x_3(s)+v(s,t)=\frac{1}{m}\int \dif t- \sum_{i=1}^{2}\int \frac{x_i'(s) \,g_{i3}(x_1(s),x_2(s))}{m^2\,U(s)^2}\dif s, \quad m\neq 0,
\]
whose does not depend of the lift of $\tilde{\gamma}$. Furthermore, by construction, the induced metric is given by $g_\psi=\dif s^2+U(s)^2\dif t^2$.

As a consequence of the discussion above, we have the following result:

\begin{taggedtheorem}{B}\label{Generalized-Bour-2}
Let $\{\psi_m(s,t)\}_m$ be the one-parameter family of $G_X$-invariant surfaces of $N$ isometric to $M$ given by Theorem~\ref{Generalized-Bour}. Then each surface of $\{\psi_m(s,t)\}_m$ can be explicitly determined in terms of $m$ and the metric $\dif \sigma^2$.
\end{taggedtheorem}

Therefore, the generalized Bour's theorem for invariant surfaces follows from Theorem~\ref{Generalized-Bour} and Theorem~\ref{Generalized-Bour-2}.

\subsection{Helicoidal surfaces revisited}

In the following, we apply our integration procedure to obtain the Bour's theorem for helicoidal surfaces in $\mathbb{R}^3$, and more generally, in BCV spaces.

\begin{example}[Helicoidal surfaces in $\mathbb{R}^3$]\label{helicoidal-R3}
In \cite{Bour-1863}, Bour presents in Theorem~II, p. 82, a two-parameter family of helicoidal surfaces isometric to a given helicoidal surface in $\r^3$ generated by a graph $z = \lambda(x)$ in the $xz$-plane, where $(x,y,z)$ denotes the cartesian coordinates on $\mathbb{R}^3$. Afterwards, do Carmo and Dajczer described all helicoidal surfaces of constant mean curvature in $\r^3$ by using the Bour's result. In order to prove the Bour's theorem for helicoidal surfaces in $\mathbb{R}^3$, let
\[
X=y\frac{\partial}{\partial x}-x\frac{\partial}{\partial y}+a\,\frac{\partial}{\partial z}, \ \text{with} \ a\in\r,
\]
be the infinitesimal generator of $G_X = \{\phi_v\}$, $v\in\mathbb{R}$, the helicoidal isometries subgroup of $\mathbb{R}^3$ with $z$-axis and pitch $a$, where
\[
\phi_v(x,y,z) = \left( x\cos v+y\sin v,y\cos v-x\sin v,z+a\,v \right).
\]
We assume that $a\neq 0$ and then the Killing vector field $X$ does not vanish on $\mathbb{R}^3$. Consider the following change of coordinates for the $X$-adapted coordinates $(x_1,x_2,x_3)$:
\begin{equation*}
\left\{\begin{aligned}
& x = x_{1} \cos x_{3} + x_{2} \sin x_{3},\\
& y = x_{2} \cos x_{3} - x_{1} \sin x_{3},\\
& z = a\, x_{3}.
\end{aligned}\right.
\end{equation*}
In these coordinates, the Euclidean metric $g$ is given by
\begin{equation*}
g = \dif  x_{1}^2
+ \dif  x_{2}^2+(x_1^2+x_2^2+a^2)\dif x_3^2
+2\left(x_2\dif x_1-x_1\dif x_2\right)\otimes\dif x_3,
\end{equation*}
the Killing vector field $X$ takes the form $X=\frac{\partial}{\partial x_3}$ and the helicoidal isometries
subgroup $G_X$ is
\[
G_X=\{\phi_v : \mathbb{R}^3\to\mathbb{R}^3 : \phi_v(x_1,x_2,x_3) = (x_1,x_2,x_3+v)\}.
\]
Let $V=\{(x_1,x_2,x_3)\in\mathbb{R}^3 : x_1>0\}.$ The coordinates on $V/G_X$ are provided by the complete set of $X$-invariant functions $\{\xi_1(x_1,x_2)=x_1,\xi_2(x_1,x_2)=x_2\}$, the Riemannian submersion is $\pi(x_1,x_2,x_3)=(x_1,x_2)$, and the quotient metric is
\[\tilde{g} = \frac{1}{x_1^2+x_2^2+a^2}\left((x_1^2+a^2)\dif x_1^2+2x_1 x_2 \dif x_1\otimes\dif x_2+(x_2^2+a^2)\dif x_2^2\right),\]
up to identification of $\xi_i$ with $x_i$.

The volume function of the principal orbit is $\omega(x_1,x_2) = \sqrt{x_1^2+x_2^2+a^2}$ and, by a direct computation, we can see that $\theta(x_1,x_2)=x_2/x_1$ is a solution of \eqref{system-f}. Thus $\tilde{\Psi}:V/G_X\to\mathbb{R}^2$, given by $\tilde{\Psi}(x_1,x_2)=(\omega(x_1,x_2),\theta(x_1,x_2))$, has rank 2, and consequently $x_1$ and $x_2$ are determined in terms of $(\omega,\theta)$ by
\begin{equation*}
\left\{\begin{aligned}
&x_1(\omega,\theta) =\sqrt{\frac{\omega^2-a^2}{1+\theta^2}},\\
&x_2(\omega,\theta) = \theta \sqrt{\frac{\omega^2-a^2}{1+\theta^2}},
\end{aligned}\right.
\end{equation*}
and with respect to the coordinates $(\omega,\theta)$, the quotient metric $\tilde{g}$ takes the form
\[
\tilde{g}=\frac{\omega^2}{\omega^2-a^2}
\dif \omega^2+\frac{a^2(\omega^2-a^2)}{\omega^2(1+\theta^2)^2}\dif \theta^2.
\]
By considering the system \eqref{system}, we have that $\omega(s)=\pm m U(s)$ and $\theta(s)$ is solution of
\[
\frac{\theta'(s)}{1+\theta(s)^2}=
\pm\frac{m\,U(s)\sqrt{m^2U(s)^2\,(1-m^2U'(s)^2)-a^2}}{a\,(m^2U(s)^2-a^2)}
\]
that is, $\theta(s) = \tan\left(\lambda(s)/a\right)$, with
\[
\lambda(s) = \pm\int \frac{m\,U(s)\sqrt{m^2U(s)^2(1-m^2U'(s)^2)-a^2}}{m^2U(s)^2-a^2}\dif s.
\]
Therefore, any profile curve $\widetilde{\gamma}$ of a helicoidal surface in $\mathbb{R}^3$ can be locally parametrized, with respect to the coordinates $(x_1,x_2)$ on $V/G_X$, by
\[
\widetilde{\gamma}(s) = \left(\sqrt{m^2U(s)^2-a^2}\cos\frac{\lambda(s)}{a}, \sqrt{m^2U(s)^2-a^2} \sin \frac{\lambda(s)}{a}\right).
\]
Moreover, if
$\gamma(s) = (x_1(s),x_2(s),x_3(s))$ is a lift of $\widetilde{\gamma}$,
we get that
\begin{align*}
x_3(s)+v(s,t)&= \frac{1}{m}\int\dif t\pm\int
\frac{\sqrt{m^2U(s)^2(1-m^2U'(s)^2)-a^2}}{a\, m\, U(s)}\dif s,
\end{align*}
consequently, \eqref{psi(s,t)} implies that any helicoidal surface in $\mathbb{R}^3$ can be locally parametrized in natural parameters, with respect to the $X$-adapted coordinates, by
\begin{equation*}
\psi(s,t) =\left(\sqrt{m^2U(s)^2-a^2}\cos\tfrac{\lambda(s)}{a}, \sqrt{m^2U(s)^2-a^2}\sin\tfrac{\lambda(s)}{a}, x_3(s)+v(s,t)\right).
\end{equation*}
In particular, we remark that, after changing for the cartesian coordinates $(x,y,z)$, for $x_3(s) = \arctan\theta(s)$ we have that
\begin{equation*}
\psi(s,t) = \left(\rho(s)\cos v(s,t),-\rho(s)\sin v(s,t),\lambda(s)+a\,v(s,t)\right),
\end{equation*}
where
\begin{equation}\label{system-1}
\left\{\begin{aligned}
\rho(s) &=\sqrt{m^2U(s)^2-a^2},\\
\lambda(s) &= \varepsilon\int \frac{m\,U(s)\sqrt{m^2U(s)^2(1-m^2U'(s)^2)-a^2}}{m^2U(s)^2-a^2}\dif s,\\
v(s,t) &=\frac{1}{m}\int \dif t -\varepsilon\int \frac{a\sqrt{m^2\,U(s)^2(1-m^2U'(s)^2)-a^2}}{m\,U(s)(m^2U(s)^2-a^2)}\dif s,
\end{aligned}\right.
\end{equation}
and $\varepsilon = \pm 1$, which for $\varepsilon = 1$ corresponds the Bour's theorem for helicoidal surfaces in $\mathbb{R}^3$  presented in \cite[Theorem~II, p. 82]{Bour-1863} and \cite[Lemma~2.3]{doCarmo-Dajczer-82}.
\end{example}

\begin{example}[Helicoidal surfaces in BCV spaces]\label{helicoidal-BCV} In \cite{Caddeo-Onnis-Piu-2022}, the second and the third authors studied helicoidal surfaces in Bianchi-Cartan-Vranceanu (BCV) spaces. These spaces consist of all three-dimensional simply-connected homogeneous manifolds whose group of isometries has dimension $4$ or $6$, except for those of constant negative sectional curvature (for the classification, see \cite{Bianchi-1897,Bianchi-1928,Cartan-1951,Piu-88,Vranceanu-57}). For sake of clarity, maintaining the notations used in \cite{Caddeo-Onnis-Piu-2022}, the BCV spaces can be represented in a concise form by the following two-parameter family of metrics
\begin{equation*}\label{1.1}
g_{\kappa,\tau} = \frac{\dif r^2}{B^2}
+r^2\left(\frac{1+\tau^2 r^2}{B^2}\right)\dif\vartheta^2+\dif z^2 - 2\frac{\tau r^2}{B}\dif\vartheta\otimes \dif z\,,
\end{equation*}
for $\kappa, \tau \in \r$ and  $B=1+\frac{\kappa}{4}r^2$  positive, where $(r,\vartheta,z)$ are cylindrical coordinates on $\mathbb{R}^3$. In the following we denote by $\N_{\kappa,\tau}$ the open subset of $\r^3$ where the metrics $g_{\kappa,\tau}$ are defined.
In a similar way to the previous example, we obtain the Bour's theorem for helicoidal surfaces in $\N_{\kappa,\tau}$, omitting long computations that appear in this case. Let
\[X=\frac{\partial}{\partial \vartheta}+a\,\frac{\partial}{\partial z}, \ \text{with} \ a\in\r,\]
be the infinitesimal generator of $G_X = \{\phi_v\}$, $v\in\mathbb{R}$, the helicoidal isometries subgroup of $\N_{\kappa,\tau}$ with $z$-axis and pitch $a$, where
\[
\phi_v(r,\vartheta,z) = \left( r, \vartheta+v,z+a\,v\right).
\]
We assume that $a\neq 0$, and we consider the following change of coordinates for the $X$-adapted coordinates $(x_1,x_2,x_3)$:
\begin{equation*}
\left\{\begin{aligned}
r &= \sqrt{x_1^2+ x_2^2},\\
\vartheta &= x_3+\arctan \frac{x_2}{x_1},\\
z &= a\, x_3.
\end{aligned}\right.
\end{equation*}
In these coordinates, the family of metrics $g_{\kappa,\tau}$ is given by
\begin{multline*}
g_{\kappa,\tau} = \frac{1+{\tau}^{2} x_{2}^{2}}{B^2} \dif  x_{1}^2
+ \frac{1+{\tau}^{2} x_{1}^{2}}{B^2}\, \dif  x_{2}^2+\frac{C^2+x_1^2+x_2^2}{B^2}\,\dif x_3^2\\
- \frac{2 x_1 x_2 \tau^2}{B^2}\,\dif x_1\otimes\dif x_2
+\frac{2(\tau C-1)}{B^2}\left(x_2\dif x_1-x_1\dif x_2\right)\otimes\dif x_3,
\end{multline*}
where $B=1+\frac{\kappa}{4}(x_1^2+x_2^2)$ is positive on $V$, and $C = a B - (x_1^2+x_2^2)\tau$, for some appropriate neighborhood $V\subset\N_{\kappa,\tau}$. Moreover, the Killing vector field $X$ takes the form
$X = \frac{\partial}{\partial x_3}$. Since $\omega(x_1,x_2) = \sqrt{C^2+x_1^2+x_2^2}/B$ is a volume function of the principal orbit, in an analogous way, we can see that $\theta(x_1,x_2)=x_2/x_1$ is a solution of \eqref{system-f} and that $x_1$ and $x_2$ are determined in terms of $(\omega,\theta)$ by
\begin{equation*}
\left\{\begin{aligned}
&x_1^2+x_2^2 = \frac{4(\omega^2-a^2)}{(1+\sqrt{\Delta})^2-4\tau^2\omega^2} ,\\
&x_2 = x_1 \theta,
\end{aligned}\right.
\end{equation*}
where $\Delta(\omega) = (1-2a\tau)^2+(4\tau^2-\kappa)(\omega^2-a^2)$.
With respect to the coordinates $(\omega,\theta)$, the quotient metric $\tilde{g}$ takes the form
\begin{multline*}
\tilde{g}=\frac{\omega^2(1-2a\tau+\sqrt{\Delta})^2}{\Delta(\omega^2-a^2)
((1+\sqrt{\Delta})^2-4\tau^2\omega^2)}
\dif\omega^2\\ +\frac{a^2(\omega^2-a^2)((1+\sqrt{\Delta})^2-4\tau^2\omega^2)}{\omega^2(1+\theta^2)^2(1-2a\tau+\sqrt{\Delta})^2
}\dif \theta^2.
\end{multline*}
By considering the system \eqref{system}, we have that $\omega(s)=\pm m U(s)$ and $\theta(s)$ is solution of
\[
\frac{\theta'}{1+\theta^2}=
\pm\frac{mU(4+\kappa \rho^2)}{4a\rho^2}\sqrt{\rho^2-\frac{ m^4 U^2 U'^2\,(4+\kappa \rho^2)^2}{16\Delta}}
\]
that is, $\theta(s) = \tan\left(\lambda(s)/a\right)$, with
\[
\lambda(s) = \pm\int \frac{m\,U(4+\kappa \rho^2)}{4\rho^2}\sqrt{\rho^2-\frac{ m^4 U^2 U'^2\, (4+\kappa \rho^2)^2}{16\Delta}}\dif s,
\]
where $$\rho(s) = \sqrt{x_1^2\,( m U(s),\theta(s))+x_2^2\,( m U(s),\theta(s))},\qquad \Delta(s)=\Delta( m U(s)).$$
Therefore, any profile curve $\widetilde{\gamma}$ of a helicoidal surface in $\N_{\kappa,\tau}$ can be locally parametrized, with respect to the coordinates $(x_1,x_2)$ on $V/G_X$, by
\[
\widetilde{\gamma}(s) = \left(\rho(s)\,\cos\frac{\lambda(s)}{a}, \rho(s)\,\sin \frac{\lambda(s)}{a}\right).
\]
Moreover, if
$\gamma(s) = (x_1(s),x_2(s),x_3(s))$ is a lift of $\widetilde{\gamma}$,
we get that
\begin{multline*}
x_3(s)+v(s,t)= \frac{1}{m}\int\dif t\\
\pm\int \frac{\tau(a\kappa-4\tau)\rho^2+4(a\tau-1)}{amU(4+\kappa \rho^2)}
\sqrt{\rho^2-\frac{ m^4 U^2 U'^2 (4+\kappa \rho^2)^2}{16\Delta}}\dif s,
\end{multline*}
consequently, \eqref{psi(s,t)} implies that any helicoidal surface in $\N_{\kappa,\tau}$ can be locally parametrized in natural parameters, with respect to the $X$-adapted coordinates, by
\begin{equation*}
\psi(s,t) =\left(\rho(s)\cos\frac{\lambda(s)}{a}, \rho(s)\sin\frac{\lambda(s)}{a}, x_3(s)+v(s,t)\right).
\end{equation*}
In particular, we remark that, after changing for the cylindrical coordinates $(r,\vartheta,z)$, for $x_3(s) = -\arctan\theta(s)$ we have that
\begin{equation*}
\psi(s,t) = \left(\rho(s),v(s,t),-\lambda(s)+a\,v(s,t)\right),
\end{equation*}
where
\begin{equation}\label{system-2}
\left\{\begin{aligned}
\rho(s) &=2\sqrt{\frac{m\,U^2-a^2}{(1+\sqrt{\Delta})^2-4\tau^2m^2U^2} }, \\
\lambda(s) &= \varepsilon\int \frac{mU(4+\kappa \rho^2)}{4\rho^2}\sqrt{\rho^2-\frac{m^4 U^2 U'^2(4+\kappa \rho^2)^2}{16\Delta}}\dif s,\\
v(s,t) &=\frac{1}{m}\int \dif t-\varepsilon\int \frac{(4\tau-a\kappa)\rho^2-4a}{4mU \rho^2}
\sqrt{\rho^2-\frac{m^4 U^2 U'^2(4+\kappa \rho^2)^2 }{16\Delta}}\dif s,
\end{aligned}\right.
\end{equation}
$\Delta(s) = (1-2a\tau)^2+(4\tau^2-\kappa)(m^2U(s)^2-a^2)$ and $\varepsilon = \pm 1$, which for $\varepsilon=-1$ corresponds the Bour's theorem for helicoidal surfaces in the BCV spaces presented in \cite[Theorem~2]{Caddeo-Onnis-Piu-2022}.
\end{example}

\begin{remark} We note that the formulas \eqref{system-1} and \eqref{system-2} also hold when $a=0$, and in these cases we obtain a rotational surface that is isometric to a given helicoidal surface. This fact generalizes the well-known result that the helicoid and the catenoid belong to a one-parameter family of isometric surfaces in $\mathbb{R}^3$. We point out that when $\tau=\kappa=0$, the helicoidal motion described in Example~\ref{helicoidal-BCV} is in the opposite direction to the one described in Example~\ref{helicoidal-R3}, therefore some signs change in the formulas of $\lambda(s)$ and $v(s,t)$ when we compare \eqref{system-1} and \eqref{system-2}.
\end{remark}

\begin{remark} In the given examples, we prescribed $x_3(s)$ in order to obtain the identical expressions as presented in \cite{Bour-1863,doCarmo-Dajczer-82} and \cite{Caddeo-Onnis-Piu-2022}, since in these works the helicoidal surfaces were parametrized by $\psi(u,v)=\phi_v(\rho(u),0,\varepsilon\,\lambda(u))$.
\end{remark}

\bibliographystyle{amsplain}
\footnotesize{
		\bibliography{references}}
\end{document}